\documentclass[11pt,a4paper]{amsart}

\usepackage{epsfig}
\usepackage{graphicx}
\usepackage[usenames,dvipsnames]{xcolor}
\usepackage[hyphens]{url}
\usepackage[colorlinks=true,linkcolor=Blue,citecolor=Green,unicode,urlcolor=Black]{hyperref}
\usepackage{amsmath,amsthm,amsfonts,amssymb}
\usepackage{nicefrac}
\usepackage{flexisym}
\usepackage[T1]{fontenc}
\usepackage[utf8]{inputenc}
\usepackage[alwaysadjust]{paralist}
\usepackage{bm}
\usepackage{latexsym}
\usepackage{enumitem}
\setlist[enumerate]{leftmargin=25pt}
\setlist[itemize]{leftmargin=25pt}

\newtheorem{thm}{Theorem}[section]
\newtheorem{lemma}[thm]{Lemma}
\newtheorem{prop}[thm]{Proposition}
\newtheorem{cor}[thm]{Corollary}
\newtheorem{conj}[thm]{Conjecture}

\theoremstyle{definition}

\newtheorem*{rem}{Remark}

\theoremstyle{definition}
\newtheorem{defn}[thm]{Definition}

\DeclareMathOperator{\inter}{int}

\def\R{\mathbb{R}}

\def\Z{\mathbb{Z}}
\def\Q{\mathbb{Q}}

\def\N{\mathbb{N}}

\DeclareMathOperator{\vol}{vol}

\DeclareMathOperator{\lin}{lin}

\DeclareMathOperator{\flt}{Flt}
\newcommand{\lgp}{LGP}

\newcommand{\ol}{\overline}

\newcommand{\supp}{\ensuremath{\operatorname{supp}}}

\DeclareMathOperator{\image}{\ensuremath{\operatorname{im}}}
\DeclareMathOperator{\bbm}{bbm}

\newcommand{\norm}[1]{{\left\|{#1}\right\|}}

\newcommand{\ent}[1]{{\left[{#1}\right]}}
\newcommand{\bra}[1]{{\left({#1}\right)}}
\newcommand{\abs}[1]{{\left|{#1}\right|}}

\newcommand{\set}[1]{{\left\{{#1}\right\}}}

\newcommand{\al}{\ensuremath{\alpha}}
\newcommand{\be}{\ensuremath{\beta}}

\newcommand{\de}{\ensuremath{\delta}}

\newcommand{\eps}{\ensuremath{\varepsilon}}

\newcommand{\La}{\ensuremath{\Lambda}}
\newcommand{\la}{\ensuremath{\lambda}}

\newcommand{\bears}{\begin{eqnarray*}}
\newcommand{\eears}{\end{eqnarray*}}

\newcommand{\mc}[1]{\ensuremath{\mathcal{#1}}}

\newcommand{\ZZ}{\ensuremath{\mathbb{Z}}}

\newcommand{\QQ}{\ensuremath{\mathbb{Q}}}
\newcommand{\RR}{\ensuremath{\mathbb{R}}}

\newcommand{\TT}{\ensuremath{\mathbb{T}}}

\newcommand{\sm}{\ensuremath{\setminus}}
\newcommand{\ssq}{\ensuremath{\subseteq}}

\newcommand{\vn}{\ensuremath{\varnothing}}

\newcommand{\half}[1]{\ensuremath{\frac{1}{2}\bs 1_{#1}}}

\newcommand{\bs}[1]{\ensuremath{\mathbf{#1}}}
\providecommand{\abs}[1]{\lvert#1\rvert}
\newcommand{\bsm}{\ensuremath{\boldsymbol}}

\numberwithin{equation}{section}

\begin{document}

\title[Covering radii, view-obstructions, and lonely runners]{On the covering radius of lattice zonotopes and its relation to view-obstructions and the lonely runner conjecture}

\author{Matthias Henze}
\thanks{The first author is supported by the Freie Universit\"at Berlin within the Excellence Initiative of the German Research Foundation.}
\address{Institut f\"ur Mathematik, Freie Universit\"at Berlin, Arnimallee 2, 14195 Berlin, Germany}
\email{matthias.henze@fu-berlin.de}

\author{Romanos-Diogenes Malikiosis} 
\thanks{The second author is supported by a Postdoctoral Fellowship from the Alexander von Humboldt Foundation.}
\address{Technische Universit\"at Berlin, Institut f\"ur Mathematik,
Sekretariat MA 4-1,
Stra{\ss}e des 17. Juni 136,
D-10623 Berlin, Germany}
\email{malikios@math.tu-berlin.de}




\begin{abstract}
The goal of this paper is twofold; first, show the equivalence between certain problems in geometry, such as view-obstruction and billiard ball motions, with
the estimation of covering radii of lattice zonotopes. Second, we will estimate upper bounds of said radii by virtue of the Flatness Theorem.
These problems are similar in nature with the famous lonely runner conjecture.
\end{abstract}

\maketitle

\section{Introduction}

The purpose of this article is to exhibit and utilize the equivalence of certain geometric problems in different settings: {\bf (a)} billiard ball motions inside a cube avoiding an inner cube,
{\bf (b)} lines in a multidimensional torus avoiding a smaller ``copy'' of the torus, {\bf (c)} views unobstructed by a lattice arrangement of cubes, and {\bf (d)} 
covering radii of lattice zonotopes.

The equivalence of the first three problems has been shown in the works of Wills~\cite{willslrc}, Cusick~\cite{cusickviewob}, 
and Schoenberg~\cite{Sch76}, among others; the equivalence to estimating covering radii of zonotopes is the novelty here.
The latter interpretation gives the possibility to use techniques from discrete geometry and the geometry of numbers in order to tackle these problems.

In a nutshell the four settings are related as follows: A billiard ball motion inside the cube $[0,1]^m$ can be unfolded into a line in the torus $\TT^m=[0,1)^m$, by reflecting appropriately its pieces between the boundary of the cube.
Then, through periodization of this configuration we obtain a lattice arrangement of lines in the space $\R^m$. So, if the billiard ball motion intersects an inner cube,
say $[\eps,1-\eps]^m$, then the corresponding line in $\TT^m$ will intersect a smaller ``copy'' of the torus, and the line in $\R^m$ will intersect a lattice arrangement of cubes.
An equivalent condition to the latter case is having a cube intersecting a lattice arrangement of lines in~$\R^m$.
Then, under an appropriate projection, we get a zonotope intersecting a lattice.
In order to get bounds on the size of the cubes under question, we will chiefly work in the zonotope setting.
The other three settings and their equivalences have been investigated in some detail before.
The keen reader might recognize that when the line in question passes through the origin we essentially deal with the \emph{lonely runner problem}.

In order to make all this precise, some notation is in order.
For standard notions in convex geometry and the theory of lattices, we refer the reader to the textbooks of Gruber~\cite{gruber2007convex} and Martinet~\cite{martinet2003perfect}, respectively.
A billiard ball motion inside the unit cube is denoted by $\bbm(\bsm u_0,\bsm\al)$, where $\bsm\al$ shall denote its initial direction and $\bsm u_0$ its starting point.
There is a linear subspace~$V_{\bsm\al}$ of~$\R^m$ that uniquely corresponds to every such $\bsm\al\in\R^m$ (see Subsection \ref{subsect_zonotopes_vo}). 
The orthogonal projection $C_m|V_{\bsm\al}$ of the unit cube $C_m=[0,1]^m$ onto $V_{\bsm\al}$ is a zonotope with vertices in $\ZZ^m|V_{\bsm\al}$.
Next, take an invertible linear map $T:V_{\bsm\al}\to\R^n$, for which $T(\ZZ^m|V_{\bsm\al})=\ZZ^n$, and denote the zonotope $T(C_m|V_{\bsm\al})$ by $Z_{\bsm\al}$.
We further let $\bs1_m=(1,\ldots,1)^\intercal$ be the all-one-vector in $\R^m$.

The discussed equivalences can now be summarized as follows.

\begin{thm}\label{equivform}
 Let $\bsm{u}_0\in\R^m$, $\bsm{\alpha}\in\R^m$, $0\leq\eps\leq1/2$, and $n=\dim(V_{\bsm{\alpha}})$.
The following statements are equivalent.
\begin{enumerate}[label={\bf (V\arabic*)},leftmargin=35pt]
 \item The $\bbm(\bsm{u}_0,\bsm{\al})$ in $[0,1]^m$ intersects $[\eps,1-\eps]^m$.
 \item The line $\left\{\bsm{u}_0+t\bsm{\al} \mid t\in\RR\right\}/\,\Z^m$ in $\TT^m$ intersects $[\eps,1-\eps]^m$.
 \item The view from $\bsm{u}_0$ with direction $\bsm{\al}$ is obstructed by $[\eps,1-\eps]^m+\ZZ^m$. 
 \item $\left((1-2\eps)Z_{\bsm{\alpha}}-\bar{\bsm{u}}_0\right)\cap\Z^n\neq\vn$, where $\bar{\bsm{u}}_0=T(\bsm{u}_0-\eps\bs1_m|V_{\bsm\al})$.
\end{enumerate}
\end{thm}

It should be noted that property {\bf(V4)} does not depend on the choice of the map $T$, as long as it satisfies $T(\ZZ^m|V_{\bsm\al})=\ZZ^n$.
Furthermore, when {\bf(V4)} fails, we have a zonotope avoiding a lattice, and through Khinchin's
\emph{Flatness Theorem} \cite{khinchin1948a}, we obtain an estimate on the largest possible $\eps$ that one can choose in Theorem~\ref{equivform} under natural constraints on~$\bsm\al$.
Since our methods are efficient in the last setting, we supply the text with the relevant tools and definitions. 
In Section~\ref{sect_rationally_uniform} it will be apparent why we further restrict the direction vectors $\bsm\al$, requiring that they may be \emph{rationally uniform}:

\begin{defn}
Let $\bsm\al\in\RR^m$ and write $\dim_\Q(\bsm\al)$ for the dimension of the $\Q$-vector space generated by the coordinates of $\bsm\al$.
Then, $\bsm\al$ is called {\em rationally uniform} if every $\dim_\Q(\bsm\al)$ coordinates of $\bsm\al$ are linearly independent over~$\Q$.
If $\bsm\al$ is rationally uniform, then $\bbm(\bsm{u}_0,\bsm\al)$ is also called such.
\end{defn}

For $\bsm u,\bsm v\in\R^n$, let $[\bsm{u},\bsm{v}]$ denote the line segment with endpoints $\bsm{u}$ and~$\bsm{v}$.
A \emph{zonotope} $Z\ssq\R^n$ is the Minkowski sum of line segments $[\bsm{a}_i,\bsm{b}_i]$, $1\leq i\leq m$, where the vectors $\bsm{b}_i-\bsm{a}_i$ generate the whole space $\R^n$.
Here, we shall only consider lattice zonotopes, that is, zonotopes whose vertices lie on a lattice~$\La$.
So, we consider $Z=\sum_{i=1}^{m}[\bsm{0},\bsm{z}_i]$, where $\bsm{z}_1,\dotsc,\bsm{z}_m\in\La$
generate $\R^n$.
Also, we may require that every $n$ vectors from the generators of the zonotope form a basis of $\R^n$; we prove that these are precisely the zonotopes $Z_{\bsm\al}$ that are obtained from rationally uniform $\bsm\al$.

\begin{defn}
Let $S\subset\R^n$ be a finite subset. We say that
 $S$ is in \emph{linear general position} (\lgp), if any $n$ points in $S$ are linearly independent.
\end{defn}

Estimating the \emph{covering radius} of a zonotope will provide us with bounds for the largest possible $\eps$ described above. This quantity is defined generally for convex bodies as follows.

\begin{defn}
Let $K\ssq\R^n$ be a convex body, and let $\La$ be a lattice. Then, the \emph{covering radius} of $K$ with respect to $\La$, denoted by $\mu(K,\La)$, is the smallest positive real number $\rho$
for which the translates of $\rho K$ by $\La$ cover the entire space $\R^n$, that is,
\[\mu(K,\La)=\inf\{\rho>0\mid\rho K+\La=\R^n\}.\]
For $\Lambda=\Z^n$, we abbreviate $\mu(Z)=\mu(Z,\Z^n)$.
\end{defn}

\begin{figure}[t]
\hfill\includegraphics[height=4cm]{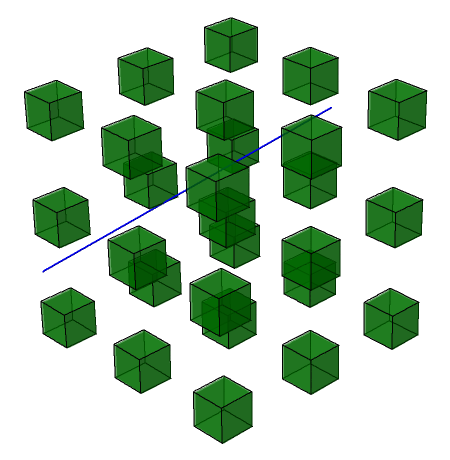}
\hfill\includegraphics[height=4cm]{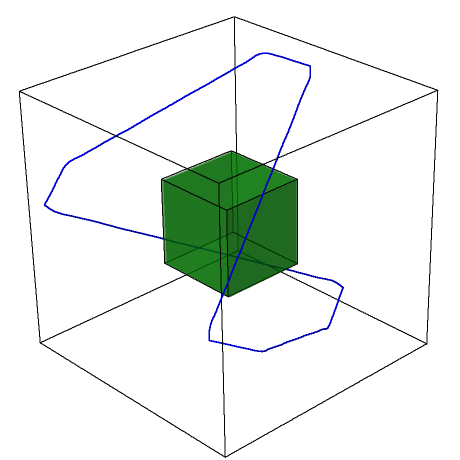}
\hfill\includegraphics[height=3.5cm]{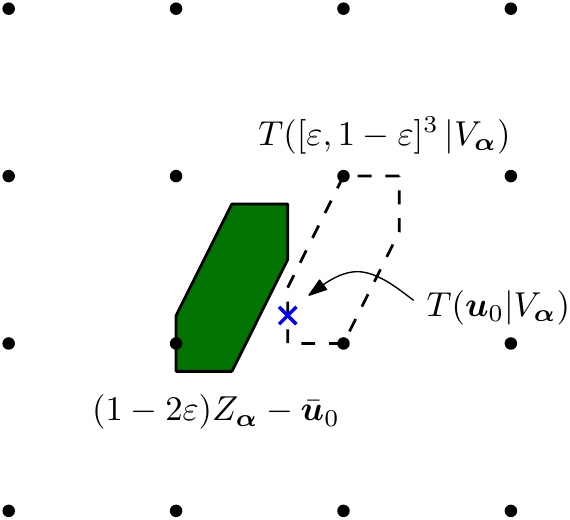}
\hfill
\caption{Illustrating {\bf (V3)}, {\bf (V1)}, and {\bf (V4)} for the parameters ${\boldsymbol u}_0=(1/6,1/2,5/6)^\intercal$, ${\boldsymbol\al}=(1,2,1)^\intercal$, and $\eps=1/3$.}
\end{figure}

We can now state our main result with respect to the covering radius of lattice zonotopes.

\begin{thm}
  \label{thmCoveringRadiusAsymptotics}
Let $S=\{\bsm{z}_1,\ldots,\bsm{z}_m\}\ssq\Z^n$ and let $Z=\sum_{i=1}^m[\bsm{0},\bsm{z}_i]$ be the lattice zonotope generated by~$S$. If $S$ is in \lgp, we have
\[\mu(Z) \leq \frac{c\,n\log{n}}{m-n+1},\]
for some absolute constant $c>0$.

Moreover, there is an $S\subset\Z^n$, $|S|=m$, in \lgp\ such that $\mu(Z) \geq \frac1{m-n+1}$.
\end{thm}

The logarithmic factor in the above estimate is most likely not needed and it has its roots in the application of the Flatness Theorem to our problem.
It is commonly believed that the \emph{flatness constant} is of order $O(n)$ rather than the currently known bound $O(n\log{n})$ (see the discussion in Section~\ref{sect_flatness}).
A beautiful result of Schoenberg~\cite{Sch76} regarding billiard ball motions inside the unit cube is the following.

\begin{thm}[Schoenberg~\cite{Sch76}]\label{Sch}
Every nontrivial billiard ball motion inside the unit cube $[0,1]^m$ intersects the cube $[\eps,1-\eps]^m$ if and only if $\eps\leq 1/(2m)$.
The number of such motions touching the cube for $\eps=1/(2m)$ is essentially finite.
\end{thm}

A nontrivial billiard ball motion is one that is not contained in a translate of a coordinate hyperplane.
As we shall see below, these motions correspond exactly to initial directions $\bsm\al$ every coordinate of which is nonzero.
An example of a billiard ball motion that intersects the cube $[\frac1{2m},1-\frac1{2m}]^m$ only in the boundary is given by the parameters
\[\bsm u_0=\bra{0,\frac{1}{m},\ldots,\frac{m-1}{m}}^\intercal\quad\text{and}\quad\bsm\al=\bs 1_m=(1,\ldots,1)^\intercal.\]
We later reformulate Schoenberg's result to the statement that the inequality $\mu(Z)\leq n/(n+1)$ holds for every lattice zonotope $Z\subseteq\R^n$ that is generated by $n+1$ vectors in \lgp.
This interpretation motivates the investigation of billiard ball motions whose initial directions have restricted rational dependencies.
Based on the equivalence of {\bf(V1)}-{\bf(V4)}, we define, for any $n\leq m$,
\begin{align*}
\eps(n,m)=\sup\big\{\eps\geq0 \mid &\,{\bf (Vi)}\text{ hold for any }\bsm u_0\in\R^m\text{ and any rationally}\\
&\text{ uniform }\bsm\al\in(\R\sm\set{0})^m\text{ with}\,\dim_{\Q}(\bsm\al)\geq m-n\big\}.
\end{align*}

The trivial cases are $\eps(0,m)=1/2$ and $\eps(m,m)=0$, whereas Theorem~\ref{Sch} translates into $\eps(m-1,m)=1/(2m)$.
We show in Section~\ref{subsect_rationally_uniform} that $\eps(n,m)=\frac12\left(1-\sup \mu(Z)\right)$, where the supremum is taken over all lattice zonotopes $Z\subseteq\R^k$ generated by~$m$ vectors of $\Z^k$ in \lgp, and where $k\leq n$.
Therefore, Theorem~\ref{thmCoveringRadiusAsymptotics} translates into the bounds
\begin{align}
\frac{m-O(n\log{n})}{2(m-n+1)}\leq\eps(n,m)\leq\frac{m-n}{2(m-n+1)}.\label{eqnEpsAsympBounds}
\end{align}
The fact that the logarithmic factor in the lower estimate of $\eps(n,m)$ might not be needed, as well as the implied constant in the case covered by Theorem~\ref{Sch} leads us to formulate:

\begin{conj}\label{mainconj}
 Let $S=\set{\bsm{z}_1,\dotsc,\bsm{z}_m}\ssq\Z^n$ and let $Z$ be the lattice zonotope generated by $S$. If $S$ is in \lgp, then $\mu(Z)\leq \frac{n}{m}$. 
\end{conj}

\begin{rem}
This is equivalent to $\eps(n,m)\geq\frac{m-n}{2m}$. By virtue of Theorem~\ref{equivform}, for $\bsm\al\in\R^m$ with $\dim_{\Q}(\bsm\al)\geq m-n$,
Conjecture \ref{mainconj} has the following equivalent forms:
\begin{itemize}
 \item Every rationally uniform $\bbm(\bsm{u}_0,\bsm\al)$ intersects $\ent{\frac{m-n}{2m},\frac{m+n}{2m}}^m$.
 \item For every $\bsm{u}_0\in\R^m$ and every rational uniform $\bsm\al\in\R^m$, the line $\set{\bsm{u}_0+t\bsm\al \mid t\in \R}/\,\Z^m$ in $\TT^m$ intersects $\ent{\frac{m-n}{2m},\frac{m+n}{2m}}^m$.
 \item From every point $\bsm u_0\in\R^m$, the view with rationally uniform direction~$\bsm\al$ is obstructed by $\ent{\frac{m-n}{2m},\frac{m+n}{2m}}^m+\Z^m$.
\end{itemize}
\end{rem}

The paper is organized as follows.
In the next section, we familiarize the reader with Schoenberg's terminology concerning billiard ball motions, albeit in a more general setting, and we rephrase Theorem~\ref{Sch}
in the context of lines in a multidimensional torus.
Finally we pass to lattice zonotopes through orthogonal projections, thus establishing the equivalence of the properties {\bf(V1)}--{\bf(V4)} in Theorem~\ref{equivform}.

In Section~\ref{sect_rationally_uniform}, we restrict our attention to rationally uniform vectors, justifying the definition of $\eps(n,m)$ above.
We show that these vectors are associated with lattice zonotopes generated by vectors in LGP and vice versa.
We briefly discuss the \emph{associated zonotope} of $Z_{\bsm\al}$, originally studied by
Shephard~\cite{shephard1974combinatorial}, and establish the monotonicity properties of~$\eps(n,m)$.

As the $\eps(n,m)$ can be expressed in terms of covering radii of lattice zonotopes generated by vectors in \lgp, we apply Khinchin's Flatness Theorem in order to produce nontrivial bounds;
Theorem~\ref{thmCoveringRadiusAsymptotics} is thus proved in Section~\ref{sect_flatness}. 

In the last section, we draw some connections to the Lonely Runner Problem, where we state with Corollary~\ref{corEquivLRP} the analogous statement to Theorem~\ref{equivform} in this setting.
We also rephrase this problem in terms of zonotopes
centered at a point of order $2$ modulo $\ZZ^n$ having a lattice point, with the hope that it could be useful towards establishing better bounds than the existing ones
(Conjecture \ref{conjZonotopalLRC}).
Aside from that, we prove that it suffices to consider integer velocities in the Lonely Runner Conjecture.
This reduction was originally stated by Wills~\cite{willslrc}, thereafter taken for granted, until a proof appeared in~\cite{sixrunners}, which however depends on solving the Lonely Runner Conjecture in lower dimensions.
In Lemma~\ref{LRCintegerReduction} we prove the reduction to integer velocities unconditionally. 
Finally, we show that the Lonely Runner Conjecture implies a more refined statement, namely Conjecture~\ref{strongLRC}, where we take the dimension of the $\Q$-span of the velocities into account.

\section{Billiard ball motions, multidimensional tori, and zonotopes}
\label{sect_theorem_equivs}

\subsection{Billiard ball motions}
\label{subsectBBM}

In~\cite{Sch76}, Schoenberg defines a billiard ball motion inside a cube as rectilinear and uniform and it is reflected in the usual way
when striking any of the cube's faces. Since the boundary of a cube is not smooth, a little care should be taken when this motion hits the boundary of the cube in lower-dimensional faces.

Let $\bsm{u}_0\in[0,1]^m$ and let $\bsm{\al}=(\al_1,\dotsc,\al_m)\in\R^m$.
Initially, the billiard ball motion starting at $\bsm{u}_0$ with inital direction $\bsm{\al}$ has the form $\bsm{u}_0+t\bsm{\al}$, $t\geq0$.
Assume that $t_0$ is the first instance when this motion hits the boundary of the standard cube $[0,1]^m$, say, at the relative interior of the facet $x_j=\eps_j$ for $j\in J\ssq\set{1,2,\dotsc,m}$, where $\eps_j=0$ or $1$.
Then, the motion is reflected and follows the path $\bsm{u}_0+t_0\bsm{\al}+(t-t_0)\bsm{\al}'$ for $t>t_0$ (until it hits the boundary again), where $\al'_j=-\al_j$ for $j\in J$ and  $\al'_j=\al_j$ otherwise.
It is clear that such a motion is completely determined by the starting point $\bsm{u}_0$ and 
the initial direction $\bsm{\al}$.

\begin{defn}
A \emph{billiard ball motion} starting at $\bsm{u}_0$ and with initial direction $\bsm{\al}$, that is reflected naturally as described above, is denoted by $\bbm(\bsm{u}_0,\bsm{\al})$. We call such a motion \emph{nontrivial} when all the coordinates
of $\bsm{\al}$ are nonzero, and \emph{rationally uniform}, when $\bsm\al$ is. 
\end{defn}

If $\bsm{v}_0$ is the symmetric point of $\bsm{u}_0+t_0\bsm{\al}$ with respect to $\half m$,
then it is clear from the above that the path $\bsm{v}_0-(t-t_0)\bsm{\al}'$ is symmetric to the
path $\bsm{u}_0+t_0\bsm{\al}+(t-t_0)\bsm{\al}'$ with respect to $\half m$.
Hence, for every $\eps\geq0$, the $\bbm(\bsm{u}_0,\bsm{\al})$ avoids the cube $[\eps,1-\eps]^m$, if and only if the line $\bsm{u}_0+t\bsm{\al}$, $t\in\R$, in the torus $\TT^m=\R^m/\Z^m$ avoids the same cube.
In the latter case, the coordinates are taken $\bmod\,1$.
Furthermore, the latter case happens if and only if the line $\set{\bsm{u}_0+t\bsm{\al} \mid t\in\R}$ avoids the set $[\eps,1-\eps]^m+\ZZ^m$.
This however can be stated as a view-obstruction property, namely that the view from~$\bsm{u}_0$ with
direction $\bsm{\al}$ is not obstructed by $[\eps,1-\eps]^m+\ZZ^m$.
In summary, these considerations prove the equivalence of the statements {\bf (V1)}, {\bf (V2)}, and {\bf (V3)} in Theorem~\ref{equivform}.

By Theorem~\ref{Sch}, when we restrict $\bsm\al$ to have no coordinate equal to zero, the infimum of such~$\eps$ is equal to $1/(2m)$.
An equivalent statement is that the $l^{\infty}$-distance from $\half m$
to any line in $\TT^m$ not parallel to a coordinate hyperplane is at most $\frac{m-1}{2m}$.
Note that $\TT^m$ inherits the $l^{\infty}$-distance from $\R^m$ as a quotient space.

%

\subsection{Lines in a multidimensional torus}\label{multitorus}

Now that the connection between billiard ball motions and lines in a torus has been established,
we want to determine the shape of a line in a multidimensional torus, or equivalently, the shape of its periodization into
the whole space $\R^m$ by the standard lattice~$\ZZ^m$.
So, for $\bsm{\al}\in\R^m$, we want to describe $\ZZ^m+\R\bsm{\al}$.
Ideas similar to those discussed below have been briefly elaborated on in~\cite[Lem.~8]{sixrunners}, and to a greater extent by the second author in~\cite{CyclotomicPyjama} with respect to a different problem.

We define the lattice
\[\La_{\bsm{\al}}:=\left\{\bsm{l}\in\ZZ^m\mid\bsm{l}\cdot\bsm{\al}=0\right\}\]
and the set
\[\mathcal{E}_{\bsm{\al}}:=\left\{\bsm{\xi}\in\R^m\mid\bsm{l}\cdot\bsm{\xi}\in\ZZ,\forall\bsm{l}\in\La_{\bsm{\al}}\right\}.\]
For a first description of $\Z^m+\R\bsm\al$, we need a special case of Kronecker's approximation theorem.

\begin{thm}[Kronecker \cite{kron1885approx}]\label{thmKronecker}
Let $\bsm\al,\bsm\be\in\RR^k$. For every $\eps>0$ there is an integer $q\in\Z$ and a vector $\bsm p\in\ZZ^k$ such that
\[\norm{q\bsm\al-\bsm p-\bsm\be}_{\infty}<\eps,\]
if and only if for every $\bsm r\in\ZZ^k$ with $\bsm r\cdot\bsm\al\in\ZZ$ we also have $\bsm r\cdot\bsm\be\in\ZZ$.
\end{thm}

\begin{lemma}\label{mainlemma}
Let $\bsm{\al}\in\R^m$. Then, $\mathcal{E}_{\bsm{\al}}$ is the closure of $\ZZ^m+\R\bsm{\al}$. 
\end{lemma}
\begin{proof}
Inclusion is obvious: let $\bsm{\xi}\in\ZZ^m+\R\bsm{\al}$ be arbitrary. So, $\bsm{\xi}=\bsm{\nu}+x\bsm{\al}$ for some $\bsm{\nu}\in\ZZ^m$ and $x\in\R$.
Now, for arbitrary $\bsm{l}\in\La_{\bsm{\al}}$, we have
\[\bsm{l}\cdot\bsm{\xi}=\bsm{l}\cdot\bsm{\nu}+x(\bsm{l}\cdot\bsm{\al})=\bsm{l}\cdot\bsm{\nu}\in\ZZ,\]
since, by definition, $\La_{\bsm{\al}}\ssq\ZZ^m$.
Hence, $\bsm{\xi}\in\mathcal{E}_{\bsm{\al}}$.

Now let $\bsm{x}=(x_1,\dotsc,x_m)\in\mathcal{E}_{\bsm{\al}}$ be arbitrary.
If $\bsm{\al}=\bsm{0}$, then obviously $\mathcal{E}_{\bsm{\al}}=\ZZ^m$ and there is nothing more to prove.
Thus, we assume that $\bsm{\al}\neq\bsm{0}$, and without loss of generality we may assume that $\al_m\neq0$.
Moreover, we can assume that $\al_m=1$, since both $\mathcal{E}_{\bsm{\al}}$ and $\La_{\bsm{\al}}$ are invariant under multiplication of $\bsm{\al}$ by a nonzero number.

We show that we can approximate $\bsm{x}$ by elements of $\ZZ^m+\R\bsm{\al}$ as close as we want.
It suffices to prove that the sequence 
\[(x_m+s)\bsm{\al}-\bsm x=((x_m+s)\al_1-x_1,\dotsc,(x_m+s)\al_{m-1}-x_{m-1},s)^\intercal,\, s\in\ZZ\]
has terms arbitrarily close to $\ZZ^m$, or equivalently, the sequence
\[(s\al_1,\dotsc,s\al_{m-1})^\intercal,\, s\in\ZZ\]
has terms arbitrarily close to $(x_1-x_m\al_1,\dotsc,x_{m-1}-x_m\al_{m-1})^\intercal+\ZZ^{m-1}$. Consider the projection $\pi:\R^m\rightarrow\R^{m-1}$
that ``forgets'' the last coordinate.
In view of Kronecker's Theorem~\ref{thmKronecker}, the closure of the subgroup of $\TT^{m-1}$ generated by 
\[\pi(\bsm{\al})+\ZZ^{m-1}=(\al_1,\dotsc,\al_{m-1})^\intercal+\ZZ^{m-1}\]
is the set of all $(a_1,\dotsc,a_{m-1})^\intercal+\ZZ^{m-1}$ for which $l_1a_1+\dotsb+l_{m-1}a_{m-1}\in\ZZ$ whenever $l_1,\dotsc,l_{m-1}\in\ZZ$ and $l_1\al_1+
\dotsb+l_{m-1}\al_{m-1}\in\ZZ$.
In other words, it suffices to prove that
\[\pi(\bsm{l})\cdot(x_1-x_m\al_1,\dotsc,x_{m-1}-x_m\al_{m-1})^\intercal\in\ZZ,\]
whenever $\bsm{l}\in\La_{\bsm{\al}}$.
So, when $\bsm{l}\in\La_{\bsm{\al}}$, by definition of $\bsm x\in\mathcal{E}_{\bsm{\al}}$ we have $\bsm{l}\cdot\bsm{x}\in\ZZ$.
Hence, $\bsm{l}\cdot(\bsm{x}-x_m\bsm{\al})\in\ZZ$. The latter is obviously equal to the desired inner product, completing the proof.
\end{proof}

\begin{rem}\rm
When the coordinates of $\bsm{\al}$ are linearly independent over $\QQ$, then $\La_{\bsm{\al}}=\{\bsm{0}\}$
and $\mathcal{E}_{\bsm{\al}}=\R^m$.
Therefore, by Lemma~\ref{mainlemma}, $\ZZ^m+\R\bsm{\al}$ is dense in~$\R^m$, or equivalently, the set
\[\left\{\left(\{x\al_1\},\{x\al_2\},\dotsc,\{x\al_m\}\right)\mid x\in\R\right\}\]
is dense in $\TT^m$, where $\{y\}=y-\lfloor y \rfloor$ denotes the fractional part of $y\in\R$.

In particular, this means that $\eps(0,m)=1/2$ as claimed in the introduction.
\end{rem}

\subsection{Zonotopes and view-obstructions}\label{subsect_zonotopes_vo}

We now finish the proof of Theorem~\ref{equivform} by providing the details of the zonotopal description of the view-obstruction problem under consideration.
Our arguments are based on a more illuminating description of $\mc E_{\bsm\al}$.
Define $V_{\bsm\al}$ to be the linear span of $\La_{\bsm\al}$, and let $\La^\star_{\bsm\al}=\{\bsm{x}\in V_{\bsm\al} \mid \bsm{l}\cdot\bsm{x}\in\Z, \forall\bsm{l}\in\Lambda_{\bsm\al}\}$ be the
dual lattice of $\La_{\bsm\al}$ inside~$V_{\bsm\al}$, where $V_{\bsm\al}$ is now considered as an inner product subspace of $\R^m$, with respect to the standard inner product.

\begin{prop}\label{lattarr}
For every $\bsm\al\in\R^m$ holds $\mc E_{\bsm\al}=\La^\star_{\bsm\al}\oplus V^{\perp}_{\bsm\al}$.
\end{prop}
\begin{proof}
Let $\bsm z\in \mc E_{\bsm\al}$ be arbitrary, and let $\bsm z=\bsm x+\bsm y$, where $\bsm x\in V_{\bsm\al}$ and $\bsm y\in V^{\perp}_{\bsm\al}$.
It suffices to prove that $\bsm l\cdot\bsm x\in\ZZ$, for all $\bsm l\in\La_{\bsm\al}\subseteq V_{\bsm\al}$.
For any such~$\bsm l$, we have $\bsm l\cdot\bsm y=0$, hence $\bsm l\cdot\bsm x=\bsm l\cdot\bsm z$, and
by definition $\bsm l\cdot\bsm z\in\ZZ$, thus proving $\mc E_{\bsm\al}\ssq\La^\star_{\bsm \al}\oplus V^{\perp}_{\bsm\al}$.
 
For the reverse inclusion, let $\bsm x+\bsm y=\bsm z\in\La^\star_{\bsm\al}\oplus V^{\perp}_{\bsm\al}$ be arbitrary with $\bsm x\in \La^\star_{\bsm \al}$ and $\bsm y\in V^{\perp}_{\bsm\al}$.
Also, let $\bsm l\in\La_{\bsm\al}\subseteq V_{\bsm\al}$ be arbitrary.
By definition, $\bsm l\cdot \bsm y=0$ and $\bsm l\cdot \bsm x\in\ZZ$, therefore $\bsm l\cdot \bsm z\in\ZZ$, thus proving that $z\in \mc E_{\bsm \al}$.
\end{proof}

Let $\bsm\al\in\R^m$.
By the previous result, $\mc E_{\bsm\al}$ consists of infinitely many parallel affine subspaces, in a lattice arrangement.
By Lemma~\ref{mainlemma}, the condition~{\bf (V3)} is equivalent to the following:
\begin{flalign*}
&\text{{\bf(V3)\textprime}\	\	\hspace{3cm}}\mc E_{\bsm\al}\text{ intersects }[\eps,1-\eps]^m-\bsm u_0. &
\end{flalign*}
Next, denote the unit cube $[0,1]^m$ by $C_m$ and assume that {\bf (V3)\textprime} holds for all $\bsm u_0\in\R^m$, which is certainly true when $\eps\leq1/(2m)$ according to Theorem~\ref{Sch} and the remarks in Section~\ref{subsectBBM}.
In other words, every translate of $(1-2\eps)C_m$ intersects $\mc E_{\bsm\al}$.
Having this in mind, we project everything orthogonally onto~$V_{\bsm\al}$.
First, we note that $\bsm e_i\in \mc E_{\bsm\al}$ for all $1\leq i\leq m$, as $\ZZ^m\ssq \mc E_{\bsm\al}$ by definition, where the $\bsm e_i$ are the standard basis vectors.
Then, $\mc E_{\bsm\al}|V_{\bsm\al}=\La^\star_{\bsm\al}$ by Proposition~\ref{lattarr}, and $C_m|V_{\bsm\al}$ is a zonotope generated by $\bsm e_i|V_{\bsm\al}=\bsm w_i\in \La^\star_{\bsm\al}$,
$1\leq i\leq m$.

In order to describe $C_m|V_{\bsm\al}$, let $\dim(V_{\bsm\al})=n$ and let $\set{\bsm{a}_1,\dotsc,\bsm{a}_n}$ be a basis of $\La_{\bsm\al}$ with $\bsm{a}_i=(a_{i1},\dotsc,a_{im})$, 
$1\leq i\leq n$.
Furthermore, let $\set{\bsm{a}_1^\star,\dotsc,\bsm{a}_n^\star}$ be the dual basis of $\La^\star_{\bsm\al}$, so that
$\bsm{a}_i\cdot\bsm{a}_j^\star=\de_{ij}$, where $\delta_{ij}$ is the Kronecker delta.
Since $\bsm{a}_i\cdot\bsm{w}_j=\bsm{a}_i\cdot\bsm e_j=a_{ij}$, we have
\[\bsm{w}_j=\sum_{i=1}^n a_{ij}\bsm{a}_i^\star.\]
Now, $C_m|V_{\bsm\al}=\sum_{j=1}^{m}[\bs 0,\bsm{w}_j]$, so if we apply the linear transformation $T:V_{\bsm\al}\to\R^n$ that sends the basis $\set{\bsm{a}_1^\star,\dotsc,\bsm{a}_n^\star}$ to the standard basis of $\R^n$, we find that the zonotope $Z=T(C_m|V_{\bsm\al})$ is generated by the columns of the matrix $\bsm A=(a_{ij})$, whose rows are given by $\bsm{a}_1^\intercal,\dotsc,\bsm{a}_n^\intercal$.
We also denote this zonotope by $Z_{\bsm\al}$ to stress its dependence on $\bsm\al$.
Note that $Z_{\bsm\al}$ also depends on the choice of the basis of $\Lambda_{\bsm\al}$, however, different such choices lead to \emph{unimodularly equivalent} zonotopes, that is, they are the same up to a linear transformation $\bsm U\in\Z^{n\times n}$ with $\det(\bsm U)=\pm 1$.
As the value of $\eps$ in Theorem~\ref{equivform} is invariant under unimodular transformations of the problem, we can safely ignore this dependence in the sequel.

By Proposition \ref{lattarr}, property {\bf(V3)\textprime} is equivalent to the fact that $\mc E_{\bsm\al}|V_{\bsm\al}=\La_{\bsm\al}^\star$ intersects $(1-2\eps)C_m|V_{\bsm\al}-(\bsm{u}_0-\eps\bs1_m)|V_{\bsm\al}$, which in turn is equivalent to the fact that $\ZZ^n$ has nontrivial intersection with $(1-2\eps)Z_{\bsm\al}-\bar{\bsm{u}}_0$, where $\bar{\bsm{u}}_0=T((\bsm{u}_0-\eps\bs1_m)|V_{\bsm\al})$, completing the proof of Theorem~\ref{equivform}.

\section{A case for rationally uniform directions}
\label{sect_rationally_uniform}

This section begins with an extended investigation of the zonotopes that arise in Theorem~\ref{equivform}, in fact, we see that for every direction vector $\bsm\al$ there are two lattice zonotopes that are associated to each other.
Once this is achieved, we find that there is a correspondence between rationally uniform directions and lattice zonotopes that are generated by vectors in \lgp.
Finally, we first solve the view-obstruction problem in the most general situation, before concentrating on the more illuminating case of rationally uniform directions.

\subsection{The associated zonotope of \texorpdfstring{$Z_{\boldsymbol\alpha}$}{Z\_{\textalpha}} }

We need to prepare ourselves with an auxiliary statement from linear algebra.
For abbreviation, we write $[m]=\{1,\ldots,m\}$, and $\binom{[m]}{k}$ for the family of $k$-element subsets of $[m]$.
Moreover, for any $I,J\in\binom{[m]}{k}$ and any matrix $\bsm V\in\R^{m\times m}$, we write $\bsm V_{I,J}\in\R^{k\times k}$ for the matrix that remains when striking out all rows of $\bsm V$ not indexed by $I$ and all columns of $\bsm V$ not indexed by $J$.
We shall denote the complement $[m]\sm I$ by $I^c$.

\begin{lemma}\label{lemMinorRelations}
Let $\bsm V=(\bsm v_1,\ldots,\bsm v_m)\in\R^{m\times m}$ be a matrix whose columns correspond to a basis of $\R^m$ and let $\bsm V^\star=(\bsm v^\star_1,\ldots,\bsm v^\star_m)$ be the matrix corresponding to the dual basis, that is, $\bsm v_i\cdot\bsm v^\star_j=\delta_{ij}$, for every $i,j\in[m]$.
For any $k\in[m]$ and any $I,J\in\binom{[m]}{k}$, we have
\[\det(\bsm V_{I,J})=\pm\det(\bsm V)\cdot\det(\bsm V^\star_{I^c,J^c}).\]
\end{lemma}
\begin{proof}
Our arguments are based on the theory of the exterior algebra of $\R^m$, for which we refer the reader to~\cite[Ch.~XVI]{birkhoffmaclane1979algebra}.
The main ingredient is the fact that for any matrix $\bsm M\in\R^{m\times m}$, we have
\[\det(\bsm M)\cdot \bsm e_1\wedge\ldots\wedge\bsm e_m=(\bsm M\bsm e_1)\wedge\ldots\wedge(\bsm M\bsm e_m).\]
Note also, that $\det(\bsm V_{I,J})=\pm\det(\tilde{\bsm V}_{I,J})$, where $\tilde{\bsm V}_{I,J}=\left(\bsm e_i \mid i\notin I, \bsm v_j \mid j\in J\right)$.
Applying these two identities repeatedly yields
\begin{align*}
\det(\bsm V_{I,J})\cdot \bsm e_1\wedge\ldots\wedge\bsm e_m&=\pm\bigwedge_{i\notin I}\bsm e_i\wedge\bigwedge_{j\in J}\bsm v_j\\
&=\pm\det(\bsm V)\cdot \bigwedge_{i\notin I}(\bsm V^{-1})_i\wedge\bigwedge_{j\in J}\bsm e_j\\
&=\pm\det(\bsm V)\cdot\det(\bsm V^{-1}_{J^c,I^c})\cdot \bsm e_1\wedge\ldots\wedge\bsm e_m.
\end{align*}
Therefore, $\det(\bsm V_{I,J})=\pm\det(\bsm V)\cdot\det(\bsm V^{-1}_{J^c,I^c})$.
Since $\bsm V^\star=\bsm V^{-\intercal}$, we have $\det(\bsm V^{-1}_{J^c,I^c})=\det(\bsm V^{-\intercal}_{I^c,J^c})=\det(\bsm V^\star_{I^c,J^c})$, which finishes the proof.
\end{proof}

\begin{rem}
One could write down the sign that appears in Lemma~\ref{lemMinorRelations} explicitely, depending on the chosen sets $I$ and $J$.
However, we refrain here from doing so, since it is not important for our purposes.
\end{rem}

Extending the investigation from the previous section, we find that a given $\bsm\al\in\R^m$ actually gives rise to two lattice zonotopes that are associated to each other in the sense of Shephard~\cite{shephard1974combinatorial}.
In order to make this precise, let as before $\bsm A=(\bsm a_1,\ldots,\bsm a_n)^\intercal$, where $\set{\bsm a_1,\ldots,\bsm a_n}$ is a basis of $\Lambda_{\bsm\al}$.
Now, we can extend this basis to a basis $\set{\bsm a_1,\ldots,\bsm a_n,\ldots,\bsm a_m}$ of $\Z^m$ since $\Lambda_{\bsm\al}=V_{\bsm\al}\cap\Z^m$.
Taking the dual basis $\set{\bsm a_1^\star,\ldots,\bsm a_m^\star}$ thereof, we find that $\set{\bsm a_{n+1}^\star,\ldots,\bsm a_m^\star}$ is a basis of $V_{\bsm\al}^\perp\cap\Z^m$, and we write $\bsm A^\perp=(\bsm a_{n+1}^\star,\ldots,\bsm a_m^\star)^\intercal$ (see, e.g.~\cite[Ch.~1]{martinet2003perfect}).
Now, using the map $T:\R^m\to\R^n\times\R^{m-n}$ that sends $\bsm a_i^\star$ to the coordinate unit vector $\bsm e_i$, for all $i\in[m]$, we obtain the lattice zonotopes
\[Z_{\bsm\al}=T\left(C_m|V_{\bsm\al}\right)\subseteq\R^n\quad\text{and}\quad Z_{\bsm\al}^\perp=T\left(C_m|V_{\bsm\al}^\perp\right)\subseteq\R^{m-n},\]
which are generated by the columns of the matrices $\bsm A$ and $\bsm A^\perp$, respectively.

It is well-known that a zonotope $Z=\sum_{i=1}^m[\bsm0,\bsm z_i]\subseteq\R^n$ can be dissected into translates of the parallelepipeds $Z_I=\sum_{i\in I}[\bsm0,\bsm z_i]$, $I\in\binom{[m]}{n}$ (see, e.g.~\cite{shephard1974combinatorial}).
The volume of the parallelepipeds that $Z_{\bsm\al}$ and $Z_{\bsm\al}^\perp$ are composed of are related as follows.

\begin{prop}\label{propVolParalleps}
Let $\bsm\al\in\R^m$.
Then, for every $I\in\binom{[m]}{n}$, we have
\[\vol_n((Z_{\bsm\al})_I)=\vol_{m-n}((Z_{\bsm\al}^\perp)_{I^c}).\]
In particular, if $\bsm\al\in\Z_{>0}^m$ with $\gcd(\alpha_1,\ldots,\alpha_m)=1$, then for every $i\in[m]$,
\[\alpha_i=\vol_{m-1}((Z_{\bsm\al})_{[m]\sm\set{i}}).\]
\end{prop}
\begin{proof}
The volume of a parallelepiped is given by the absolute value of the determinant of its generators.
So, we can just apply Lemma~\ref{lemMinorRelations} to the matrices $\bsm V=(\bsm a_1,\ldots,\bsm a_m)$ and $\bsm V^\star=(\bsm a_1^\star,\ldots,\bsm a_m^\star)$, and obtain
\[\vol_n((Z_{\bsm\al})_I)=|\det(\bsm V_{I,[n]})|=|\det(\bsm V)|\cdot|\det(\bsm V^\star_{I^c,[m]\sm[n]})|=\vol_{m-n}((Z_{\bsm\al}^\perp)_{I^c}),\]
where we used that $\det(\bsm V)=\pm1$.

In the special case that $\bsm\al\in\Z^m_{>0}$ with $\gcd(\alpha_1,\ldots,\alpha_m)=1$, we have $V_{\bsm\al}^\perp=\lin\{\bsm\al\}$ and hence $V_{\bsm\al}^\perp\cap\Z^m=\Z\bsm\al$.
Therefore, we have $\bsm A^\perp=\bsm\al^\intercal$ and thus $Z_{\bsm\al}^\perp=\sum_{i=1}^m[0,\alpha_i]$, implying the claim.
\end{proof}

\subsection{Rationally uniform vectors}
\label{subsect_rationally_uniform}

Next, we study the situation when~$\bsm\al$ is \emph{rationally uniform}.
We show that $\bsm\al$ possesses this property, if and only if the vectors generating $Z_{\bsm\al}$ are in \lgp.
Two auxiliary statements are needed.

\begin{prop}
Let $\bsm{\al}\in\R^m$ and let $\Lambda_{\bsm\al},V_{\bsm\al}$ be defined as above, in particular, let $n=\dim(V_{\bsm\al})$.
Then, $\dim_\Q(\bsm\al)=m-n$.
\end{prop}
\begin{proof}
Let $f:\Q^m\to\R$ be the $\Q$-linear map defined by $f(\bsm\ell)=\bsm\ell\cdot\bsm\al$.
Clearly, $V_{\bsm\al}=\lin(\ker(f))$ and $\dim_\Q(\bsm\al)=\dim_\Q(\image(f))$.
Hence, by the rank-nullity theorem we get $m=\dim_\Q(\ker(f))+\dim_\Q(\image(f))=n+\dim_\Q(\bsm\al)$.
\end{proof}

\begin{prop}\label{propMinorsAndSubalphas}
Let $\bsm\al\in\R^m$ and let $\{\bsm a_1,\ldots,\bsm a_n\}$ be a basis of $\Lambda_{\bsm\al}$.
Define $\bsm A\in\Z^{n\times m}$, as before, as the matrix whose $i$th row is given by $\bsm a_i$.
For a subset $I\subseteq[m]$, let $\bsm A_I$ be the submatrix of $\bsm A$ consisting of those columns of~$\bsm A$ that are indexed by~$I$, and define $\bsm\al_I$ analogously.
Then, if $|I|=n$,
\[\det(\bsm A_I)\neq 0\quad\Longleftrightarrow\quad\dim_\Q(\bsm\al_{I^c})=m-n.\]
\end{prop}
\begin{proof}
Let $L_I=\lin\{\bsm e_i \mid i\in I\}$.
Since by definition $V_{\bsm\al}^\perp=\ker(\bsm A)$ and $\dim(V_{\bsm\al}^\perp)+\dim(L_I)=m$, the following equivalences hold:
\begin{align*}
\det(\bsm A_I)\neq 0 &\Longleftrightarrow \forall\bsm\ell\in\R^m,\supp(\bsm\ell)\subseteq I,\bsm A\cdot\bsm\ell=0 \Longrightarrow \bsm\ell=\bsm 0\\
&\Longleftrightarrow V_{\bsm\al}^\perp\cap L_I=\{\bsm 0\}\\
&\Longleftrightarrow V_{\bsm\al}\cap L_{I^c}=V_{\bsm\al}\cap L_I^\perp=\{\bsm 0\}.
\end{align*}
As $\dim(\Lambda_{\bsm\al})=\dim(V_{\bsm\al})$ and $\lin_\Q(\Lambda_{\bsm\al})$ is dense in $V_{\bsm\al}$, the latter is equivalent to the fact that every $\bsm\ell\in\Z^m$ with $\bsm\ell\cdot\bsm\al=0$ and $\supp(\bsm\ell)\subseteq I^c$ is already the zero-vector, that is, $\bsm\ell=\bsm 0$.
In other words, the coordinates of $\bsm\al$ that are indexed by $I^c$ are linearly independent over~$\Q$, which means that $\dim_\Q(\bsm\al_{I^c})=m-n$ as claimed.
\end{proof}

\begin{cor}\label{cor_ratunif_lgp}
\ \\[-1em]
\begin{enumerate}[label=\roman*)]
 \item The vector $\bsm\al\in\R^m$ is rationally uniform if and only if the zonotope~$Z_{\bsm\al}$ is generated by vectors in \lgp.
 \item The zonotope $Z_{\bsm\al}$ is generated by vectors in \lgp\, if and only if the associated zonotope $Z_{\bsm\al}^\perp$ is.
\end{enumerate}
\end{cor}
\begin{proof}
\romannumeral1): By the analysis in Subsection~\ref{subsect_zonotopes_vo}, we can take a map $T$ such that~$Z_{\bsm\al}$ is generated by the columns of the matrix $\bsm A$, as described in Proposition~\ref{propMinorsAndSubalphas}.
The conclusion of the latter proves the desired fact.

Part \romannumeral2) is a direct consequence of Lemma~\ref{lemMinorRelations} and Proposition~\ref{propMinorsAndSubalphas}.
\end{proof}

Shephard~\cite{shephard1974combinatorial} proves Corollary~\ref{cor_ratunif_lgp}~\romannumeral2) combinatorially in the language of \emph{cubical zonotopes}.

So far, we have only seen zonotopes attached to a vector $\bsm\al\in\R^m$, and Corollary~\ref{cor_ratunif_lgp} gives us a characterization of such zonotopes that are generated by vectors in \lgp.
This construction can be reversed: Suppose we have a lattice zonotope $Z\ssq\R^n$, generated by $m$ vectors in \lgp.
Let~$\bsm A$ be the $n\times m$ matrix whose columns correspond to these vectors.
Taking the rows of $\bsm A$, we obtain a basis of an $n$-dimensional lattice $\La$ in $\R^m$.
Let~$V$ be the space spanned by $\La$, and let $V^{\perp}$ be its orthogonal complement.
Now, one can choose a vector $\bsm\al\in V^{\perp}$, that is not orthogonal to any lattice vector, except from those of $\La$.
Indeed, if a vector $\bsm\be$ does not belong to~$V$, then $\bsm\be^{\perp}\cap V^{\perp}$ is a codimension-one subspace of $V^{\perp}$.
So, $\bsm\al$ needs to avoid countably many hyperplanes of $V^{\perp}$, and as is well-known, this is indeed possible, because countably many hyperplanes cannot cover the entire space.
We conclude by noting that in this case, $\La=\La_{\bsm\al}$, and the zonotopes $Z$ and $Z_{\bsm\al}$ are unimodularly equivalent.

If, for $\bsm\al\in\R^m$, the conditions {\bf (V1)}-{\bf (V4)} in Theorem~\ref{equivform} hold for all $\bsm u_0\in\R^m$ and some $\eps>0$, then we can deduce that any translate of $(1-2\eps)Z_{\bsm\al}$ contains a
lattice point, or in other words,
\begin{align*}
\mu(Z_{\bsm\al})&\leq 1-2\eps.
\end{align*}
This interpretation of the covering radius is classical; we refer, for instance, to~\cite[\S 13]{gruberlekker1987geometry}.
Therefore, Theorem~\ref{equivform} implies that
\begin{align}
\sup\mu(Z) &= 1-2\eps(n,m),\label{eqnEpsSupMu}
\end{align}
where the supremum is taken over all lattice zonotopes $Z\ssq\R^k$ with $m$ generators in \lgp, and where $k\leq n$.

In the case $n=m-1$, Theorem~\ref{Sch} gives us $\eps(m-1,m)=1/(2m)$, and thus, a reinterpretation of Schoenberg's result is as follows.

\begin{thm}\label{zonocodim1}
 Let $Z=\sum_{i=1}^{n+1}[\bs 0,\bsm z_i]$, where $\set{\bsm z_1,\dotsc,\bsm z_{n+1}}\ssq\ZZ^n$ is in \lgp. Then $\mu(Z)\leq \frac{n}{n+1}$. Equality is attained for
 $\bsm z_i=\bsm e_i$, $1\leq i\leq n$, $\bsm z_{n+1}=\bs 1_n$.
\end{thm}

The case $n=1$ can also be treated easily in terms of this zonotopal approach.
Indeed, a one-dimensional lattice zonotope $Z$ that is generated by~$m$ vectors in \lgp\ is just an interval with integer endpoints and length at least~$m$.
Its covering radius is just the inverse of its length, thus $\mu(Z)\geq1/m$, which in view of~\eqref{eqnEpsSupMu} translates into $\eps(1,m)\leq(m-1)/(2m)$.
An example showing that this is actually an identity is given by any $\bsm\al\in\R^m$ such that $\La_{\bsm\al}=\Z\cdot\bsm 1_m$, because this yields an interval $Z_{\bsm\al}$ of length exactly~$m$.
It is easy to see that any $\bsm\al=(\al_1,\ldots,\al_{m-1},-\alpha_1-\ldots-\alpha_{m-1})^\intercal$ with $\set{\alpha_1,\ldots,\alpha_{m-1}}$ linearly independent over~$\Q$ will do the job.
Hence, $\eps(1,m)=(m-1)/(2m)$.

Together with the trivial cases $n=0$ and $n=m$ the above considerations show that Conjecture~\ref{mainconj} holds for all $n\in\set{0,1,m-1,m}$.
The question remains in the intermediate cases, where we obtain a weaker bound, namely Theorem~\ref{thmCoveringRadiusAsymptotics} (see Section~\ref{sect_flatness} for details).




\subsection{The monotonicity of the function \texorpdfstring{$\eps(n,m)$}{\textepsilon(n,m)}}

Remember that one of our main interests is to determine, or at least estimate, the supremum among all $\eps\geq0$ such that the equivalent statements in Theorem~\ref{equivform} hold for every starting point $\bsm u_0\in\R^m$ and every nontrivial direction $\bsm\al\in(\R\sm\set{0})^m$ with $\dim_{\Q}(\bsm\al)=m-n$.

We first see that in this generality the problem can be solved exactly but depends only on the parameter~$n$.
To this end, let us consider the numbers
\begin{align*}
\ol\eps(n,m)=\sup\big\{\eps\geq0 \mid &\,{\bf (Vi)}\text{ holds for any }\bsm u_0\in\R^m\text{ and }\bsm\al\in(\R\sm\set{0})^m\\
&\text{ such that }\dim_{\Q}(\bsm\al)\geq m-n\big\}.
\end{align*}

\begin{prop}\label{propPropertiesBarEpsFunction}
 Let $n,m\in\N$ be such that $n\leq m$. Then,
\begin{enumerate}[label=\roman*)]
 \item $\ol\eps(n,m+1)\leq\ol\eps(n,m)$, and
 \item if $n<m$, then $\ol\eps(n,m)=\frac1{2(n+1)}$.
 \end{enumerate}
\end{prop}
\begin{proof}
\romannumeral1): For the sake of brevity let $\ol\eps=\ol\eps(n,m+1)$.
Let $\bsm{u}'_0\in\R^m$ and $\bsm{\al}'\in(\R\setminus\{0\})^m$ with $\dim_\Q(\bsm{\al}')\geq m-n$ be arbitrary.
Now, let $\bsm{u}_0=(\bsm{u}'_0,t)\in\R^{m+1}$, for some $t\in\R$, and let $\bsm{\al}=(\bsm{\al}',a)$, for some $a\in\R\setminus\{0\}$ that is rationally independent from the entries of $\bsm{\al}'$.
Note that $\dim_\Q(\bsm{\al})=\dim_\Q(\bsm{\al}')+1\geq m+1-n$ and $\bsm\al\in(\R\sm\set{0})^{m+1}$.

By definition of $\ol\eps$, we have $(\bsm{u}_0+\mc E_{\bsm\al})\cap[\ol\eps,1-\ol\eps]^{m+1}\neq\vn$.
The choice of~$a$ implies that no multiple $qa$, $q\in\Z\setminus\{0\}$, can be written as $\bsm{\ell}'\cdot\bsm{\al}'$, for some $\bsm{\ell}'\in\Z^m$.
Therefore, $\Lambda_{\bsm{\al}}=\Lambda_{\bsm{\al}'}\times\{0\}$, and hence
\begin{align*}
\pi(\mc E_{\bsm{\al}})&=\left\{\pi(\bsm{\xi}) \mid \bsm{\xi}\in\R^{m+1},\bsm{\ell}\cdot\bsm{\xi}\in\Z,\,\forall \bsm{\ell}\in\Lambda_{\bsm{\al}}\right\}\\
&=\left\{\bsm{\xi}'\in\R^m \mid \bsm{\ell}'\cdot\bsm{\xi}'\in\Z,\,\forall \bsm{\ell}'\in\Lambda_{\bsm{\al}'}\right\}=\mc E_{\bsm{\al}'},
\end{align*}
where $\pi:\R^{m+1}\to\R^m$ is the projection that forgets the last coordinate.
As a consequence we obtain that
\[(\bsm{u}'_0+\mc E_{\bsm{\al}'})\cap[\ol\eps,1-\ol\eps]^m=\pi(\bsm{u}_0+\mc E_{\bsm{\al}})\cap\pi([\ol\eps,1-\ol\eps]^{m+1})\neq\vn,\]
and hence $\ol\eps(n,m)\geq\ol\eps=\ol\eps(n,m+1)$ as desired.

\romannumeral2): In view of part~\romannumeral1), we have
\[\ol\eps(n,m)\leq\ol\eps(n,m-1)\leq\ldots\leq\ol\eps(n,n+1)\leq\eps(n,n+1)=\frac1{2(n+1)},\]
where the last equation is due to Theorem~\ref{Sch}.

For the lower bound, we consider the vector
\[\bsm\al=(\al_1,\ldots,\al_{m-n},\al_1,\ldots,\al_1)\in(\R\sm\set{0})^m,\]
where $\set{\al_1,\ldots,\al_{m-n}}$ is linearly independent over~$\Q$.
A basis of the lattice~$\Lambda_{\bsm\al}$ is given by $\{\bsm e_1-\bsm e_{m-n+1},\bsm e_1-\bsm e_{m-n+2},\ldots,\bsm e_1-\bsm e_m\}$, 
so we find that $Z_{\bsm\al}=[0,-1]^n+[\bsm0,\bsm 1_n]$.
This is exactly the same zonotope $Z_{\ol{\bsm\al}}$ that is induced by $\ol{\bsm\al}=\bsm 1_{n+1}\in\R^{n+1}$.
By the example following Theorem~\ref{Sch} we know that $\mu(Z_{\bsm\al})=\mu(Z_{\ol{\bsm\al}})=\frac{n}{n+1}$, and hence $\ol\eps(n,m)\geq\frac12(1-\mu(Z_{\bsm\al}))=\frac1{2(n+1)}$.
\end{proof}

Our intuition says that the more we restrict rational dependencies in the direction vector $\bsm\al$ the larger we can choose $\eps$ in Theorem~\ref{equivform}.
Proposition~\ref{propPropertiesBarEpsFunction} above shows that in order for this to be true, we need to impose stronger conditions than only $\dim_{\Q}(\bsm\al)\geq m-n$.
In fact, these remarks explain and justify the restriction to rationally uniform vectors in the definition of $\eps(n,m)$, and hence in the formulation of Conjecture~\ref{mainconj}.

In order to study this situation in more detail, we use the zonotopal definition of $\eps(n,m)$ provided by {\bf (V4)} together with Corollary~\ref{cor_ratunif_lgp}, that is,
\begin{align*}
\eps(n,m)=\sup\big\{\eps\geq0 \mid &\,\mu(Z,\Z^k)\leq1-2\eps\text{ for all lattice zonotopes }Z\subseteq\R^k\\
&\,\text{generated by }m\text{ vectors in \lgp\ and where }k\leq n\big\}.
\end{align*}


The following monotonicity properties of the function $\eps(n,m)$ are compatible with our conjecture that $\eps(n,m)=(m-n)/(2m)$ (see the remark after Conjecture~\ref{mainconj}).

\begin{prop}\label{propPropertiesEpsFunction}
Let $n,m\in\N$ be such that $n\leq m$. Then, we have
\begin{enumerate}[label=\roman*)]
 \item $\eps(n,m)\leq\eps(n-1,m)$ and $\eps(n,m)\leq\eps(n,m+1)$,
 \item $\eps(n,m)\leq\eps(n-1,m-1)\leq\ldots\leq\eps(1,m-n+1)=\frac{m-n}{2(m-n+1)}$.
\end{enumerate}
\end{prop}
\begin{proof}
\romannumeral1): The inequality $\eps(n,m)\leq\eps(n-1,m)$ follows directly from the definition of $\eps(n,m)$.
For the second inequality, let $Z'=\sum_{i=1}^{m+1}[\bsm 0,\bsm z_i]\subseteq\R^k$, for some $k\leq n$, be generated by lattice vectors in \lgp.
Dropping the last generator gives us a zonotope $Z=\sum_{i=1}^m[\bsm0,\bsm z_i]$, which in view of $k\leq n<m+1$ is of the same dimension as $Z'$ and generated by vectors in \lgp.
Since, clearly $Z\subseteq Z'$, this implies $\mu(Z',\Z^k)\leq\mu(Z,\Z^k)\leq1-2\eps(n,m)$.

\romannumeral2): It suffices to show that $\eps(n,m)\leq\eps(n-1,m-1)$.
Let $Z'=\sum_{i=1}^{m-1}[\bsm 0,\bsm z'_i]$ be generated by $\{\bsm z'_1,\ldots,\bsm z'_{m-1}\}\subseteq\Z^k$ in \lgp, for some $k\leq n-1$.
We define $\bsm z_m=\bsm e_{k+1}$ and $\bsm z_i=(\bsm z'_i,h_i)$, for $i\in[m-1]$, where $h_i\in\Z$ are chosen such that $\{\bsm z_1,\ldots,\bsm z_m\}$ is in \lgp.
Now, the zonotope $Z=\sum_{i=1}^m[\bsm 0,\bsm z_i]\subseteq\R^{k+1}$ projects onto $Z'$, that is, $Z'=\pi(Z)$, and clearly $\Z^k=\pi(\Z^{k+1})$, where $\pi:\R^{k+1}\to\R^k$ forgets the last coordinate.
Therefore, we have $\mu(Z',\Z^k)\leq\mu(Z,\Z^{k+1})\leq1-2\eps(n,m)$, which implies the desired inequality.
\end{proof}


%

\section{Bounds on \texorpdfstring{$\eps(n,m)$}{\textepsilon(n,m)} via the Flatness Theorem}
\label{sect_flatness}

In this section, we prove Theorem~\ref{thmCoveringRadiusAsymptotics}.
Our arguments are based on the so-called \emph{Flatness Theorem}, which states that every convex body that does not contain lattice points in its interior is necessarily flat in a lattice direction.

\begin{thm}[Khinchin~\cite{khinchin1948a}]
  \label{thmFlatness}
Let $ n \in \N $ and let $ K \subseteq \R^n $ be a convex body with $ \inter(K+\bsm t) \cap \Z^n = \vn $, for some $\bsm t\in\R^n$.
There exists a vector $ \bsm{v} \in \Z^n\setminus\{\bsm 0\}$ and a constant $ \flt(n) $ only depending on $n$ such that
\[w(K, \bsm{v}) := \max_{\bsm{x} \in K} \bsm{x}\cdot \bsm{v}  - \min_{\bsm{x} \in K} \bsm{x}\cdot \bsm{v}\le \flt(n).\]
\end{thm}

There has been quite a lot of research on estimating the dimensional constant $\flt(n)$.
We refer to the textbook by Barvinok~\cite[Ch.~VII.8]{barvinok2002acourse} for a proof of Khinchin's result, a discussion of the history of the problem, as well as for the current state of the art.
The best result to date for the class of $o$-symmetric convex bodies, that is, convex bodies $K\subseteq\R^n$ such that $K=-K$, is due to Banaszczyk~\cite{banaszczyk1996inequalities}, who proved that
\begin{align}
\text{If }K\subseteq\R^n\text{ is }o\text{-symmetric, then }\flt(n)\leq c\,n\log{n},\text{ for some }c>0.\label{eqnFlatnessBound}
\end{align}
It is generally believed that the optimal such bound is $\flt(n)\in O(n)$.

Recall that the covering radius $\mu(K)$ of $K$ is the minimal dilation $\mu>0$ such that every translate of $\mu K$ contains a point of $\Z^n$.
Therefore, the Flatness Theorem reformulates as
\begin{align}
\mu(K)w(K)\leq\flt(n),\label{eqnFlatnessReformulation}
\end{align}
where $w(K)=\min_{\bsm{v}\in\Z^n\setminus\{\bsm{0}\}}w(K,\bsm{v})$ denotes the \emph{lattice width} of $K$.
As a consequence, in order to obtain upper bounds on the covering radius $\mu(Z)$ of a lattice zonotope~$Z$, it suffices to study its lattice width.

\begin{lemma}\label{lemLatticeWidth}
Let $S=\{\bsm{z}_1,\ldots,\bsm{z}_m\}\subset\Z^n$, $m\geq n$, and let $Z=\sum_{i=1}^m[\bsm{0},\bsm{z}_i]$ be the lattice zonotope generated by~$S$.
If $S$ is in \lgp, we have \[w(Z)\geq m-n+1.\]
For every $m\geq n$ there exists such a set $S$ in \lgp\ attaining the bound.
\end{lemma}
\begin{proof}
Let $\bsm{v}\in\Z^n\setminus\{\bsm{0}\}$.
Since all vertices of $Z$ are lattice points, the width of $Z$ in direction $\bsm{v}$ is one less than the number of lattice planes parallel to $\bsm{v}^\perp=\{\bsm{x}\in\R^n \mid \bsm{v}\cdot\bsm{x}=0\}$ that intersect $Z$.
More precisely, writing $L_{\bsm{v}}=\lin\{\bsm{v}\}$, we have
\[w(Z,\bsm{v})=\abs{\left\{\bsm{y}\in\Z^n|L_{\bsm{v}} \ \big|\ (\bsm{v}^\perp+\bsm{y})\cap Z\neq\vn\right\}}-1,\]
where $\Z^n|L_{\bsm{v}}$ is the lattice that arises as the projection of $\Z^n$ onto $L_{\bsm{v}}$.
Hence, we need to estimate the number of points of $\Z^n|L_{\bsm{v}}$ in the (one-dimensional) projected zonotope $Z|L_{\bsm{v}}=\sum_{i=1}^m[0,\bsm{z}_i|L_{\bsm{v}}]$.

Since $S$ is in \lgp, the hyperplane $\bsm{v}^\perp$ contains at most $n-1$ generators of~$Z$.
Hence, there are at least $m-(n-1)$ nonzero generators of $Z|L_{\bsm{v}}$, which implies that the segment $Z|L_{\bsm{v}}$ contains at least $m-(n-1)+1$ points of $\Z^n|L_{\bsm{v}}$.
From the above considerations this yields $w(Z,\bsm{v})\geq m-n+1$, and as $\bsm v\in\Z^n\sm\set{\bsm 0}$ was arbitrary, we get $w(Z)\geq m-n+1$ as desired.

In order to show that the bound cannot be improved in general, we construct a set of generators in \lgp, that satisfies the above projection properties extremally for a particular lattice direction~$\bsm{v}$.
For $\bsm{v}=\bsm e_1$ such a set is given by
\begin{align}
S=\left\{\bsm e_1,\ldots,\bsm e_n\right\} \cup \left\{(1,\ell,\ell^2,\ldots,\ell^{n-1})^\intercal \mid \ell=1,\ldots,m-n\right\}.\label{eqnExtremalGeneratingSet}
\end{align}
The determinant formula for the Vandermonde matrix readily implies that~$S$ is indeed in \lgp.
\end{proof}

\begin{proof}[Proof of Theorem~\ref{thmCoveringRadiusAsymptotics}]
The upper bound follows by combining the Flatness Theorem in its formulation~\eqref{eqnFlatnessReformulation}, with Lemma~\ref{lemLatticeWidth} above and Banaszczyk's bound~\eqref{eqnFlatnessBound} on $\flt(n)$.
Note, that we can apply the latter since zonotopes are $o$-symmetric up to translation.

We have seen in Lemma~\ref{lemLatticeWidth} that the lattice zonotope $Z$ generated by the set $S$ in~\eqref{eqnExtremalGeneratingSet} has lattice width $w(Z)=m-n+1$.
Now, for any convex body $K\subseteq\R^n$, one has $\mu(K)w(K)\geq1$.
In fact, assuming that $K$ is scaled such that $w(K)=1$, we find a vector $\bsm v\in\Z^n\sm\set{\bsm 0}$ such that $w(K,\bsm v)=1$, which means that up to a translation $K$ is sandwiched between two parallel consecutive lattice planes orthogonal to~$\bsm v$.
Hence, there exists a translate of~$K$ whose interior does not contain lattice points, and thus $\mu(K)\geq1$.
Together with the previous observations, the covering radius of the lattice zonotope~$Z$ can now be bounded by $\mu(Z)\geq1/(m-n+1)$, as desired.
\end{proof}

\begin{rem}
The construction of lattice zonotopes $Z$ with $\mu(Z)\geq1/(m-n+1)$ in the above proof together with~\eqref{eqnEpsSupMu} shows the upper bound
\[\eps(n,m)=\frac12\left(1-\sup \mu(Z)\right)\leq\frac12\left(1-\frac1{m-n+1}\right)=\frac{m-n}{2(m-n+1)}.\]
This coincides with the bound derived by the monotonicity of the function $\eps(n,m)$ in Proposition~\ref{propPropertiesEpsFunction}.
\end{rem}

\section{A Reformulation of the Lonely Runner Conjecture}
\label{sect_LRC}

In this section, we consider the special case of billiard ball motions, or equivalently view-obstruction problems, where the starting point $\bsm u_0=\bsm0$.
This restricted variant was independently introduced by Wills~\cite{willslrc} as a problem in Diophantine approximation (see the description~\eqref{eqnDiophApprDescr} below) and by Cusick~\cite{cusickviewob} in the form of the view-obstruction formulation.
Goddyn came up with the nowadays very popular interpretation which he coined the \emph{Lonely Runner Conjecture}.
It states that if $m$ runners with nonzero constant velocities run on a circular track
of length $1$, with common starting point, then in a certain moment in time all runners are away from the starting point, having distance at least $1/(m+1)$.
This conjecture has been proven for all $m\leq6$ but is open for all other cases; see~\cite{barajasserra2008the} for the proof for $m=6$ and more background information on the 
problem\footnote{This case corresponds to seven runners, as the seventh runner is considered to have velocity equal to zero, being fixed at the starting point.}.

As a corollary to Theorem~\ref{equivform}, we may summarize the equivalent interpretations as follows.
We use the symbol $\bsm v$ instead of $\bsm\al$ in the sequel to stress that the problem only depends on the velocities of the runners.

\begin{cor}\label{corEquivLRP}
Let $\bsm v\in\R^m$, $0\leq\eps\leq 1/2$, and $n=\dim(V_{\bsm v})$.
The following statements are equivalent.
\begin{enumerate}[label={\bf (L\arabic*)},leftmargin=35pt]
 \item The $\bbm(\bsm 0,\bsm v)$ in $[0,1]^m$ intersects $[\eps,1-\eps]^m$.
 \item The line $\left\{t\bsm v \mid t\in\RR\right\}/\,\Z^m$ in $\TT^m$ intersects $[\eps,1-\eps]^m$.
 \item The view from $\bsm 0$ with direction $\bsm v$ is obstructed by $[\eps,1-\eps]^m+\ZZ^m$. 
 \item $\left((1-2\eps)Z_{\bsm v}+\bsm c\right)\cap\Z^n\neq\vn$, where $\bsm c=T(\eps\bs1_m | V_{\bsm v})$.
\end{enumerate}
\end{cor}

As before we want to determine the maximum $\eps\in[0,1/2]$ such that any of these equivalent statements hold for all velocity vectors $\bsm v$, under certain natural constraints.
Analogously to $\eps(n,m)$ and $\ol\eps(n,m)$, we therefore define
\begin{align*}
\eps_0(n,m)=\sup\{\eps\geq0 \mid &\,{\bf (Li)}\text{ holds for all rationally uniform }\bsm v\in(\R\setminus\{0\})^m\\
&\text{ with}\,\dim_\Q(\bsm v)\geq m-n\},
\end{align*}
and
\begin{align*}
\ol\eps_0(n,m)=\sup\{\eps\geq0 \mid &\,{\bf (Li)}\text{ holds }\forall\bsm v\in(\R\setminus\{0\})^m\text{ with}\,\dim_\Q(\bsm v)\geq m-n\}.
\end{align*}
The Lonely Runner Conjecture now claims that 
\[\eps_0(m-1,m)=\ol\eps_0(m-1,m)=1/(m+1),\]
and an extremal velocity vector would be given by $\bsm v=(1,2,\ldots,m)^\intercal$.
As we shall see in more detail below, in fact integral velocities alone determine $\eps_0(m-1,m)$.
On the other hand, due to a result by Czerwi{\'n}ski~\cite{czerwinski2012random}, we know that random velocity vectors allow for $\eps$ arbitrarily close to $1/2$ in Corollary~\ref{corEquivLRP}.
Since random vectors have linearly independent entries over~$\Q$, this is in the spirit of the trivial identity $\eps_0(0,m)=\ol\eps_0(0,m)=1/2$.
Using Theorem~\ref{thmCoveringRadiusAsymptotics} in the formulation~\eqref{eqnEpsAsympBounds} and similar ideas as in Propositions~\ref{propPropertiesBarEpsFunction}~\romannumeral1)
and \ref{propPropertiesEpsFunction}~\romannumeral1), we can interpolate between these two extremal situations.

\begin{cor}\label{corLRPBounds}
For any $n,m\in\N$ with $n\leq m$, we have
\begin{enumerate}[label=\roman*)]
 \item $\eps_0(m-1,m)\leq\ldots\leq\eps_0(1,m)\leq\eps_0(0,m)=\frac12$,
 \item $\eps_0(n,m)\geq\eps(n,m)\geq\frac{m-O(n\log{n})}{2(m-n+1)}$, and
 \item $\ol\eps_0(n,m)\leq\ol\eps_0(n,m+1)$.
\end{enumerate}
\end{cor}

These inequalities say that the more rational dependencies we allow in the velocity vector $\bsm v$ the smaller we have to choose $\eps$ in Corollary~\ref{corEquivLRP}.

In spite of these partial results, the main interest is of course to determine $\eps_0(m-1,m)$.
Wills~\cite{willslrc} first stated that the problem reduces to the case of nonzero velocities $v_1,v_2,\ldots,v_m\in\Z\sm\set{0}$ with $\gcd(v_1,\ldots,v_m)=1$, or equivalently, to vectors $\bsm v\in(\R\sm\set{0})^m$ with $\dim_\Q(\bsm v)=1$.
A proof of such a statement though, appeared only in Bohman, Holzman \& Kleitman~\cite[Lem.~8]{sixrunners}.
However, their lemma is not without restrictions; it states that if the conjecture is true for nonzero rational velocities in $m-1$ dimensions, then it is also true
for irrational velocities in $m$ dimensions, thus leaving only the rational case to prove in $m$ dimensions. 
We also remark that this statement is included in Corollary~\ref{corLRPBounds}~\romannumeral3) as the special case $n=m-1$.



Using the tools that we developed so far, we drop the dependence on lower dimensions, and show that the Lonely Runner Conjecture reduces to nonzero integer velocities in any
dimension unconditionally.

\begin{lemma}\label{LRCintegerReduction}
For any $m\in\N$, we have
\[\eps_0(m-1,m)=\sup\set{\eps>0 \mid \mc E_{\bsm\al}\cap[\eps,1-\eps]^m\neq\vn\text{ for all }\bsm\al\in(\ZZ\sm\set{0})^m}.\]
\end{lemma}
\begin{proof}
Restricting the attention to vectors with integral coordinates clearly cannot decrease the considered supremum, showing that $\eps_0(m-1,m)$ is less than or equal to the right hand side of the claimed identity.

In order to show the reverse inequality, it suffices to show that for every $\bsm\al\in(\R\sm\set{0})^m$ there is a $\bsm\be\in(\ZZ\sm\set{0})^m$, such that $\mc E_{\bsm\be}\ssq\mc E_{\bsm\al}$. 
In order to see this, observe that by definition $\dim(V_{\bsm\al})=\dim(V_{\bsm\al}\cap\ZZ^m)=\dim(\La_{\bsm\al})$ and analogously $\dim(V_{\bsm\al}^{\perp}\cap\ZZ^m)=\dim(V_{\bsm\al}^{\perp})>0$.
So there is a nonzero $\bsm\be\in V_{\bsm\al}^{\perp}\cap\ZZ^m$. 

We claim that in fact there is such a $\bsm\be$ in $(\ZZ\sm\set{0})^m$.
Let us assume the contrary, that is, every $\bsm\be\in V_{\bsm\al}^{\perp}\cap\ZZ^m$ has at least one coordinate equal to zero.
Suppose for the moment that there exist nonzero $\bsm\be=(\be_1,\dotsc,\be_m)$ and $\bsm\be'=(\be'_1,\dotsc,\be'_m)$ in $V_{\bsm\al}^\perp\cap\Z^m$ that have no common zero coordinate, that is, for every $j\in[m]$, either $\be_j\neq0$ or $\be'_j\neq0$.
It is easy to see that this means that the vector
\[\bsm\be+(2\abs{\be_1}+\dotsb+2\abs{\be_m})\bsm\be'\in V_{\bsm\al}^{\perp}\cap\ZZ^m\]
has no coordinate equal to zero, contradicting the assumption.
Hence, all $\bsm\be\in V_{\bsm\al}^{\perp}\cap\ZZ^m$ have a common zero coordinate, implying that a coordinate vector belongs to $V_{\bsm\al}$, and thus $\bsm\al$ has a zero coordinate, a contradiction.
 
Thus, there is a $\bsm\be\in V_{\bsm\al}^{\perp}\cap(\ZZ\sm\set{0})^m$, and by definition we have $\La_{\bsm\al}\subseteq\La_{\bsm\be}$, implying that $\mc E_{\bsm\be}\ssq\mc E_{\bsm\al}$, as desired.
\end{proof}



So, without loss of generality let $\bsm v=(v_1,\ldots,v_m)\in\ZZ^m_{>0}$ (replacing $v_i$ by $-v_i$ does not change the conclusion of the conjecture).
In view of {\bf (L2)}, we have $\eps_0(m-1,m)=1/(m+1)$ if and only if the inequalities
\begin{align}
\frac{1}{m+1}&\leq \set{tv_i}\leq \frac{m}{m+1},\quad 1\leq i\leq m\label{eqnDiophApprDescr}
\end{align}
hold for some $t\in\R$, which is Wills' original formulation given in~\cite{willslrc}.
Note that these inequalities together with Corollary~\ref{corLRPBounds}~\romannumeral3) would imply the following more general statement.

\begin{conj}\label{strongLRC}
 For any $v_1,\dotsc,v_m\in\R_{>0}$ with $\dim_{\Q}(v_1,\dotsc,v_m)\geq d$, there is some $t\in \R$, such that
\[\frac{1}{m+2-d}\leq \set{tv_i}\leq \frac{m+1-d}{m+2-d},\quad 1\leq i\leq m.\]
\end{conj}

%
%
However, we want to take a more detailed look at condition~{\bf (L4)}, that is, $\left((1-2\eps)Z_{\bsm v}+\bsm c\right)\cap\Z^n\neq\vn$, where $\bsm c=T(\eps\bs1_m | V_{\bsm v})$.
To this end, let $Z_{\bsm v}=\sum_{i=1}^m[\bsm 0,\bsm z_i]$, where $\bsm z_i=T(\bsm e_i|V_{\bsm v})\in\Z^{m-1}$ and where $T$ is the linear map from Subsection~\ref{subsect_zonotopes_vo}.
Writing $\bsm x=\frac12\sum_{i=1}^m\bsm z_i$ for the center of $Z_{\bsm v}$ and putting $\eps=1/(m+1)$, we get
\begin{align*}
(1-2\eps)Z_{\bsm v}+\bsm c&=\frac{m-1}{m+1}(Z_{\bsm v}-\bsm x)+\frac{m-1}{m+1}\bsm x+\frac1{m+1}\sum_{i=1}^mT(\bsm e_i|V_{\bsm v})\\
&=\frac{m-1}{m+1}(Z_{\bsm v}-\bsm x)+\bsm x.
\end{align*}
With the analysis above, putting $n=m-1$, we obtain the following reformulation of the lonely runner conjecture.

\begin{conj}[Reformulation of the Lonely Runner Conjecture]\label{conjZonotopalLRC}
Let~$Z$ be a zonotope generated by $n+1$ vectors of $\ZZ^n$ in LGP, and let $\bsm{x}$ be the center of $Z$.
Then,
\[\left(\bsm{x}+\frac{n}{n+2}(Z-\bsm{x})\right)\cap\ZZ^n\neq \vn.\]
\end{conj}

In the case that $\bsm x\in\ZZ^n$, there would be nothing to prove, so we can assume otherwise.
Then, $\ZZ^n$ and $\bsm x$ generate a lattice $\Lambda=\Z^n\cup(\bsm x+\Z^n)$.
Shifting everything by~$\bsm x$, so that $Z$ is $o$-symmetric with respect to the origin, our desired (nonempty) intersection becomes
\[\frac{n}{n+2}Z\cap(\Lambda\sm\ZZ^n).\]
In other words, we wish to prove that if we dilate $Z$ by a factor of at most $n/(n+2)$, we get a nontrivial point of $\Lambda$ not contained in $\ZZ^n$.
This alludes to the notion of the \emph{first restricted successive minimum}, as defined in \cite{henkthiel}:
\[\la_1(Z,\Lambda\sm\ZZ^n)=\min\set{\la\geq0 : \la Z\cap (\Lambda\sm\ZZ^n)\neq\vn}.\]
Hence, yet another reformulation of Conjecture~\ref{conjZonotopalLRC} with respect to this definition is the following:
{\em Let $Z\ssq\R^n$ be a zonotope generated by $n+1$ vectors of $\ZZ^n$ in LGP, and which is translated as to be $o$-symmetric.
Let $\Lambda$ be a lattice such that $\Lambda\sm\ZZ^n$ is a translate of $\Z^n$.
Then \[\la_1(Z,\Lambda\sm\ZZ^n)\leq\frac{n}{n+2}.\]}
In view of Schoenberg's result on the general view-obstruction problem (see Theorem~\ref{Sch}), we know that an upper bound of $n/(n+1)$ holds true.
It is quite interesting that this has not been significantly improved; the interested reader may consult~\cite{perarnauserra2016correlation} for an informative discussion of this matter.
So, even if it is difficult to prove the Lonely Runner Conjecture in the geometric setting, it is reasonable to ask whether the intersection
\[\frac{n+2-c}{n+2}Z\cap (\Lambda\sm\ZZ^n)\]
is nonempty for some absolute constant $c>1$.
Any such result would be a significant advance to the problem.


\bibliographystyle{amsplain}
\bibliography{jointbib}

\end{document}